\documentclass[12pt,a4paper,reqno]{amsart}
\allowdisplaybreaks
\usepackage{amsmath}
\usepackage{amsfonts}
\usepackage{amssymb,amsthm,amsfonts,amsthm,latexsym,enumerate,url,cases}
\numberwithin{equation}{section}
\usepackage{mathrsfs}
\usepackage{hyperref}
\hypersetup{colorlinks=true,citecolor=blue,linkcolor=blue,urlcolor=blue}
\def\pmod #1{\ ({\rm{mod}}\ #1)}

\theoremstyle{plain}
\newtheorem{theorem}{Theorem}
\newtheorem{lemma}{Lemma}

\newtheorem{proposition}{Proposition}

\theoremstyle{definition}

\usepackage{etoolbox}
\makeatletter
\patchcmd{\@settitle}{\uppercasenonmath\@title}{}{}{}
\patchcmd{\@setauthors}{\MakeUppercase}{}{}{}
\patchcmd{\section}{\scshape}{}{}{}
\makeatother

\begin{document}

\title
[{On a conjecture of Ram\'{\i}rez Alfons\'{\i}n and Ska{\l}ba III}]
{On a conjecture of Ram\'{\i}rez Alfons\'{\i}n and Ska{\l}ba III}

\author
[Y. Ding\quad  {\it and}\quad T. Komatsu] {Yuchen Ding\quad  {\it and}\quad Takao Komatsu}

\address{(Yuchen Ding) School of Mathematical Science,  Yangzhou University, Yangzhou 225002, People's Republic of China}
\email{ycding@yzu.edu.cn}
\address{(Takao Komatsu) Department of Mathematical Sciences, School of Science,  Zhejiang Sci-Tech University, Hangzhou 310018 China, People's Republic of China}
\email{komatsu@zstu.edu.cn}


\begin{abstract}
Let $1<c<d$ be two relatively prime integers. For a non-negative integer $\ell$, let $g_\ell(c,d)$ be the largest integer $n$ such that $n=c x+d y$ has at most $\ell$ non-negative solutions $(x,y)$. In this paper we prove that
$$
\pi_{\ell,c,d}\sim\frac{\pi\bigl(g_\ell(c,d)\bigr)}{2 \ell+2}\quad(\text{as}~ c\to\infty)\,,
$$
where $\pi_{\ell,c,d}$ is the number of primes $n$ having more than $\ell$ distinct non-negative solutions to $n=c x+d y$ with $n\le g_\ell(c,d)$,
and $\pi(x)$ denotes the number of all primes up to $x$.
The case where $\ell=0$ has been proved by Ding, Zhai and Zhao recently, which was conjectured formerly by Ram\'{\i}rez Alfons\'{\i}n and Ska{\l}ba.
\end{abstract}
\maketitle

\section{Introduction}\label{sec:1}

Let $c_1,\dots,c_k$ ($k\ge 2$) be a set of distinct integers with $c_i>1$ ($i=1,\dots,k$).
For a given non-negative integer $\ell$, let $S_\ell(c_1,\dots,c_k)$ (written as $S_\ell$ for shorthand), which is called the $\ell$-numerical semigroup, with $\gcd(c_1,\dots,c_k)=1$ be the set of all the elements $n$ whose number of solutions to $c_1 x_1+\cdots+c_k x_k=n$ is more than $\ell$. Then for the set of non-negative integer $\mathbb N_0$, $\mathbb N_0\backslash S_\ell$ is the set of all the elements of $n$ whose number of solutions to $c_1 x_1+\cdots+c_k x_k=n$ ($x_1,\dots,x_k\in\mathbb N_0$) is less than or equal to $\ell$. In \cite{Kom1,Kom2,KY}, the properties of the $\ell$-numerical semigroups and explicit forms of crucial numbers are discussed\footnote{In \cite{Kom1,Kom2,KY} and more references, the terminology of $p$-numerical semigroup is frequently used. However, in this paper, we mainly deal with prime numbers, so we do not use $p$ or $q$, but instead use $\ell$.}.

The set $\mathbb N_0\backslash S_\ell$ is finite if and only if $\gcd(c_1,\dots,c_k)=1$. Then, there exists the largest element $g_\ell(c_1,\dots,c_k)$, which is called the {\it $\ell$-Frobenius number}. When $\ell=0$, $g(c_1,\dots,c_k)=g_0(c_1,\dots,c_k)$ is the original Frobenius number, which is the central topic on the classically well-known linear Diophantine problem of Frobenius.
In general, to find an explicit closed form of $g_\ell(c_1,\dots,c_k)$ is very difficult for $k\ge 3$, but when $k=2$, the $\ell$-Frobenius number can be given explicitly:
for any non-negative integer $\ell$
\begin{equation}
g_\ell(c,d)=(\ell+1)c d-c-d
\label{eq:ell-frob}
\end{equation}
(see \cite{Kom1,Kom2,KY} for more general formulas and related concepts).

For integers $c,d$ with $1<c<d$ and $\gcd(c,d)=1$, let $\pi_{\ell,c,d}$ be the number of primes in
$$
\mathcal N_\ell(c,d):=\{n\in\mathbb N_0|n\le g_\ell(c,d),~n\in S_\ell(c,d)\}\,,
$$
and $\pi(x)$ denotes the number of all primes up to $x$.

Ram\'{\i}rez Alfons\'{\i}n and Ska{\l}ba \cite{RA-S} proved that for any $\varepsilon>0$, there is a constant $k_\varepsilon>0$ such that
$$
\pi_{0,c,d}\ge k_\varepsilon\frac{g_0(c,d)}{\bigl(\log g_0(c,d)\bigr)^{2+\epsilon}}\,,
$$
and conjectured the following:
\begin{equation}
\pi_{0,c,d}\sim\frac{\pi(g_0\bigl(c,d)\bigr)}{2}\quad(\text{as}~ c\to\infty)\,.
\label{ding}
\end{equation}
Ding \cite{Ding} made some progress on the above conjecture (\ref{ding}). More precisely, for all but at most
$$
O\bigl(N(\log N)^{1/2}(\log\log N)^{1/2+\varepsilon}\bigr)
$$
pairs $c$ and $d$, one has
$$
\pi_{0,c,d}=\frac{\pi\bigl(g_0(c,d)\bigr)}{2}+O\left(\frac{\pi\bigl(g_0(c,d)\bigr)}{\bigl(\log\log(c d)\bigr)^\varepsilon}\right)\,.
$$
Since
$$
\frac{\pi\bigl(g_0(c,d)\bigr)}{2}+O\left(\frac{\pi\bigl(g_0(c,d)\bigr)}{\bigl(\log\log(c d)\bigr)^\varepsilon}\right)\sim\frac{\pi\bigl(g_0(c,d)\bigr)}{2}\quad(\text{as}~ c\to\infty)
$$
and the total number of the pairs $(c,d)$ with $1<c<d$, $\gcd(c,d)=1$ and $c d\ll N$ is $\gg N\log N$, Ding's result provided an almost all version of (\ref{ding}).

Though it seemed to be out of reach \cite{RA-S}, very recently Ding, Zhai and Zhao \cite{DZZ} completely proved (\ref{ding}).

The main purpose of this paper is to show a more general result of (\ref{ding}) as follows:

\begin{theorem}
Let $\ell$ be a non-negative integer.
For integers $c$ and $d$ with $1<c<d$ and $\gcd(c,d)=1$, we have
$$
\pi_{\ell,c,d}\sim\frac{\pi\bigl(g_\ell(c,d)\bigr)}{2 \ell+2}\quad(\text{as}~ c\to\infty)\,.
$$
\label{th1}
\end{theorem}
\noindent
{\it Remark.}
When $\ell=0$, this is reduced to the main result in \cite[Theorem 1.1]{DZZ}.
The result itself is still never obvious because the speed of convergence is very slow.

\section{Preliminaries}
From now on, $c$ and $d$ will always denote two positive integers satisfying that $1<c<d$ and $\gcd(c,d)=1$. Below, the first lemma is quite trivial.

\begin{lemma}
Suppose that $c x+d y=cx'+dy'$ with $x,y,x',y'\in\mathbb N_0$, then
$$c|(y-y') \quad  \text{and} \quad d|(x-x').$$
\label{lem3}
\end{lemma}

\begin{lemma}
Suppose that the number of solutions to
$$
n=c x+d y\quad(x,y\in\mathbb N_0)
$$
is exactly $\ell$, then we have
$$
n=(\ell-1)c d+c x_0+d y_0
$$
for some $0\le x_0\le d-1$ and $0\le y_0\le c-1$.
\label{lem4}
\end{lemma}
\begin{proof}
Suppose that $\ell$ solutions are given as
$$
n=c x_1+d y_1=c x_2+d y_2=\cdots=c x_\ell+d y_\ell\,.
$$
Without loss of generality, assume that $$x_1<x_2<\dots<x_\ell.$$ So, $y_1>y_2>\dots>y_\ell$. Now, by Lemma \ref{lem3}, we have
$$
x_2=x_1+d,\quad x_3=x_1+2 d,\quad \dots,\quad x_\ell=x_1+(\ell-1)d
$$
with $0\le x_1\le d-1$ and $0\le y_\ell\le c-1$. Thus,
$$
n=c\bigl(x_1+(\ell-1)d\bigr)+d y_\ell=(\ell-1)c d+c x_1+d y_\ell.
$$
This completes the proof of Lemma \ref{lem4}.
\end{proof}

Let $\mathcal P$ be the set of primes.
\begin{lemma}
Let $\pi_{\ell,c,d}$ be defined as the one in introduction, then we have
$$
\pi_{\ell,c,d}:=\sum_{\substack{\ell c d<n\le g_\ell(c,d)\\ n=\ell c d+c x_0+d y_0\in\mathcal P\\ 0\le x_0\le d;\,0\le y_0\le c}}1\,.
$$
\label{lem5}
\end{lemma}
\begin{proof}
If $n$ has more than $\ell+1$ solutions with the form $c x+d y$ ($x,y\in\mathbb N_0$), then by Lemma \ref{lem4}, we have
\begin{equation*}
n=(\ell+k) c d+c x_0+d y_0\qquad(0\le x_0\le d-1;\,0\le y_0\le c-1)\,.
\end{equation*}
for some $k\in \mathbb{Z}^+$. For these $n$, we will have
$$
n\ge (\ell+1)cd>g_\ell(c,d).
$$
So, if $n\le g_{\ell}(c,d)$ with more than $\ell$ solutions, then $n$ has exactly $\ell+1$ solutions (expressions). Employing again Lemma \ref{lem4}, we obtain that
\begin{equation*}
n=\ell c d+c x_0+d y_0\qquad(0\le x_0\le d-1;\,0\le y_0\le c-1)\,,
\end{equation*}
provided that $n\le g_\ell(c,d)$.

Furthermore, if $n=\ell c d+c x_0+d y_0\le g_\ell(c,d)$ with
$$0\le x_0\le d \quad \text{and} \quad 0\le y_0\le c,$$
then we clearly have $x_0\le d-1$ and $y_0\le c-1$.

Now, the lemma follows from the trivial fact that $\ell cd\not\in \mathcal P$.
\end{proof}

As usual, the von Mangoldt function $\Lambda(n)$ is defined as
$$
\Lambda(n)=\begin{cases}
\log p&\text{if $n=p^\alpha$ ($\alpha>0$)};\\
0&\text{otherwise}.
\end{cases}
$$

In order to prove Theorem \ref{th1}, we shall firstly provide the following weighted version whose proof will be put on the next section.

\begin{proposition} Let $\ell$ be a non-negative integer.
For integers $c$ and $d$ with $1<c<d$ and $\gcd(c,d)=1$, we have
$$
\psi_{\ell,c,d}\sim\frac{g_0(c,d)}{2}\quad(\text{as}~ c\to\infty)\,,
$$
where
$$
\psi_{\ell,c,d}=\sum_{\substack{\ell c d<n\le g_\ell(c,d)\\ n=\ell c d+c x_0+d y_0\\ 0\le x_0\le d;\,0\le y_0\le c}}\Lambda(n)\,.
$$
\label{prp1}
\end{proposition}

\section{Proof of Proposition \ref{prp1}}
 \begin{proof}[Proof of Proposition \ref{prp1}]
Our proof follows from the argument of \cite[Theorem 1.1]{DZZ} with some adjustments. By the same shorthand in \cite{DZZ}, we write $g_0(c,d)=cd-c-d$ briefly as $g$. Throughout the proof, the integer $c$ is supposed to be sufficiently large.

By definitions of $\psi_{\ell,c,d}$, we have
\begin{align*}
\psi_{\ell,c,d}=\sum_{\substack{\ell c d<n\le g_\ell(c,d)\\ n=\ell c d+c x_0+d y_0\\ 0\le x_0\le d;\,0\le y_0\le c}}\Lambda(n)
=\sum_{\substack{n\le g\\ n=c x_0+d y_0\\0\le x_0\le d;\,0\le y_0\le c}}\Lambda(\ell c d+n)+O_\ell(\log g)\,.
\end{align*}
 For any real $\alpha$, let
\begin{align*}
f(\alpha)&=\sum_{0\le n\le g}\Lambda(\ell c d+n)e(\alpha n)\,,\\
h(\alpha)&=\sum_{0\le x\le d\atop 0\le y\le c}e\bigl(\alpha(c x+d y)\bigr)\,,
\end{align*}
where $e(t)$ denotes $e^{2\pi it}$ for any number $t$ as usual.
Then by the orthogonal relation, we have
\begin{equation}
\psi_{\ell,c,d}=\int_0^1 f(\alpha)h(-\alpha)d\alpha\,.
\label{eq:12}
\end{equation}

We are in a position to introduce the Hardy--Littlewood method to evaluate the above integral.

Let $Q<c^{1/3}$ denote a parameter depending only on $c$ and $d$ which shall be decided later. Define the major arcs to be

\begin{align*}
\mathfrak{M}(Q)=\bigcup_{1\le q\le Q}\bigcup_{\substack{1\le a\le q\\ \gcd(a,q)=1}}\left\{\alpha:\left|\alpha-\frac{a}{q}\right|\le \frac{Q}{qg}\right\}.
\end{align*}
By our assumption, we have $Q<g^{1/6}$, from which it follows trivially that the above subsets are pairwise disjoint (see e.g. \cite[Section 2]{DZZ}).
In addition, we note that
$$\mathfrak{M}(Q)\subseteq \left[\frac{1}{Q}-\frac{Q}{qg}, 1+\frac{Q}{qg}\right] \subseteq\left[\frac{Q+1}{g}, 1+\frac{Q+1}{g}\right].$$
The minor arcs are then defined to be
\begin{align*}
\mathfrak{m}(Q)=\left[\frac{Q+1}{g}, 1+\frac{Q+1}{g}\right]\setminus \mathfrak{M}(Q).
\end{align*}
From Eq. (\ref{eq:12}), it is plain that
\begin{align}
\psi_{\ell,c,d} &= \int_{\frac{Q+1}{g}}^{1+\frac{Q+1}{g}}f(\alpha)h(-\alpha)d\alpha\nonumber\\
&=\int_{\mathfrak{M}(Q)}f(\alpha)h(-\alpha)d\alpha+\int_{\mathfrak{m}(Q)}f(\alpha)h(-\alpha)d\alpha.
\label{new1}
\end{align}

\subsection{Estimates of the minor arcs}

Note that
\begin{align*}
\left|f(\alpha)\right|&=\left|\sum_{0\le n\le g}\Lambda(\ell c d+n)e(\alpha n)\right|\\
&=\left|\sum_{\ell cd\le m\le \ell cd+g}\Lambda(m)e(\alpha m)e(-\alpha\ell c d)\right|\quad(m=\ell c d+n)\\
&\le\left|\sum_{\ell cd\le m\le \ell cd+g}\Lambda(m)e(\alpha m)\right|.
\end{align*}
By a remarkable theorem of Vinogradov (see e.g. \cite[Theorem 3.1]{Vaughan}) as well as the Dirichlet approximation theorem (see e.g. \cite[Lemma 2.1]{Vaughan}), we can obtain that 
$$
\sup_{\alpha\in\mathfrak m(Q)}\left|\sum_{m}\Lambda(m)e(\alpha m)\right|\ll_\ell \frac{g(\log g)^4}{Q^{1/2}}+g^{4/5}(\log g)^4\,
$$
(see \cite[Lemma 3.1]{DZZ}), which means that
$$
\sup_{\alpha\in\mathfrak m(Q)}|f(\alpha)|\ll_\ell \frac{g(\log g)^4}{Q^{1/2}}+g^{4/5}(\log g)^4\,.
$$
By using the estimate in \cite[Lemma 3.2]{DZZ}
$$
\int_0^1|h(-\alpha)|d\alpha\ll(\log g)^2\,,
$$
we have
\begin{align}
\int_{\mathfrak m(Q)}f(\alpha)h(-\alpha)&\le\sup_{\alpha\in\mathfrak m(Q)}|f(\alpha)|\int_{\mathfrak m(Q)}|h(-\alpha)|d\alpha\nonumber\\
&\ll_\ell \left(\frac{g(\log g)^4}{Q^{1/2}}+g^{4/5}(\log g)^4\right)\int_0^1|h(-\alpha)|d\alpha\nonumber\\
&\ll_\ell \frac{g(\log g)^6}{Q^{1/2}}+g^{4/5}(\log g)^6\,.
\label{new2}
\end{align}

\subsection{Calculations of the major arc}

We now calculate the integral on the major arcs.
\begin{align}
\int_{\mathfrak M(Q)}f(\alpha)h(-\alpha)
&=\sum_{1\le q\le Q}\sum_{1\le a\le q\atop\gcd(a,q)=1}\int_{\frac{a}{q}-\frac{Q}{q g}}^{\frac{a}{q}+\frac{Q}{q g}}f(\alpha)h(-\alpha)d\alpha\nonumber\\
&=\sum_{1\le q\le Q}\sum_{1\le a\le q\atop\gcd(a,q)=1}\int_{-\frac{Q}{q g}}^{\frac{Q}{q g}}f\left(\theta+\frac{a}{q}\right)h\left(-\theta-\frac{a}{q}\right)d\theta\nonumber\\
&=\int_{-\frac{Q}{g}}^{\frac{Q}{g}}f(\theta)h(-\theta)d\theta
+\mathcal{R}\,,
\label{new3}
\end{align}
where
$$
\mathcal{R}=\sum_{1<q\le Q}\sum_{1\le a\le q\atop\gcd(a,q)=1}\int_{-\frac{Q}{q g}}^{\frac{Q}{q g}}f\left(\theta+\frac{a}{q}\right)h\left(-\theta-\frac{a}{q}\right)d\theta.
$$
We shall later prove that $\mathcal{R}$ still contributes to the `error term'.

For any real $\theta$, we have
\begin{align*}
f(\theta)&=\sum_{0\le n\le g}\Lambda(n+\ell c d)e(\theta n)\\
&=\sum_{\ell c d\le m\le g+\ell c d}\Lambda(m)e\bigl(\theta(m-\ell c d)\bigr)\quad (m=n+\ell c d)\\
&=e(-\theta\ell c d)\widetilde{f(\theta)}\,,
\end{align*}
where
$$
\widetilde{f(\theta)}=\sum_{\ell c d\le m\le g+\ell c d}\Lambda(m)e(\theta m)\,.
$$
Let $\rho(m)=\Lambda(m)-1$. Then
\begin{align*}
\widetilde{f(\theta)}-\sum_{\ell c d\le m\le g+\ell c d}e(\theta m)=\sum_{\ell c d< m\le g+\ell c d}\rho(m) e(\theta m)+O_\ell\left(\log g\right)\,.
\end{align*}
By partial summation, we have
\begin{multline}
\sum_{\ell c d< m\le g+\ell c d}\rho(m) e(\theta m)=e\big((\ell c d+g)\theta\bigr)\sum_{m\le\ell c d +g}\rho(m)\\
-e(\ell c d\theta)\sum_{m\le\ell c d}\rho(m) -2\pi i\theta\int_{\ell c d}^{\ell c d+g}\Bigg(\sum_{m\le t}\rho(m)\Bigg)e(t\theta)d t\,.
\label{new4}
\end{multline}
Using the prime number theorem, there exists some absolute constant $\kappa_1>0$ so that
$$
\sum_{m\le t}\rho(m)=\psi(t)-t\ll t e^{-\kappa_1\sqrt{\log t}}.
$$
Inserting this into Eq. (\ref{new4}), it follows that
\begin{align*}
\sum_{\ell c d<m\le g+\ell c d}\rho(m) e(\theta m)
&\ll_\ell g e^{-\kappa_1\sqrt{\log g}}+|\theta|\int_{\ell c d}^{\ell c d+g}t e^{-\kappa_1\sqrt{\log t}}d t\\
&\ll_\ell g e^{-\kappa_1\sqrt{\log g}}+|\theta|e^{-\kappa_1\sqrt{\log g}}\int_{\ell c d}^{\ell c d+g}t ~d t\\
&\ll_\ell  g (1+|\theta|g) e^{-\kappa_1\sqrt{\log g}}\,.
\end{align*}
Thus,
$$
\widetilde{f(\theta)}=\sum_{\ell c d\le m\le g+\ell c d}e(\theta m)+O_\ell\bigl(g (1+|\theta|g) e^{-\kappa_1\sqrt{\log g}}\bigr)\,.
$$
Hence,
\begin{align*}
f(\theta)&=e(-\theta\ell c d)\sum_{\ell c d\le m\le g+\ell c d}e(\theta m)
+O_\ell\bigl(g (1+|\theta|g) e^{-\kappa_1\sqrt{\log g}}\bigr)\\
&=\sum_{0<n\le g}e(\theta n)+O_\ell\bigl(g (1+|\theta|g) e^{-\kappa_1\sqrt{\log g}}\bigr)\,.
\end{align*}
By the estimates of \cite[Lemma 4.4]{DZZ}, we have
\begin{align}
\int_{|\theta|\le\frac{Q}{g}}f(\theta)h(-\theta)d\theta=\frac{g}{2}+O_\ell\left(\frac{g}{Q}(\log g)^2+g Q^2 e^{-\kappa_1\sqrt{\log g}}\right)\,.
\label{new5}
\end{align}
For $Q<c^{1/3}$, by \cite[Lemma 4.5]{DZZ} we also have
\begin{align}
\mathcal{R}=\sum_{2\le q\le Q}\sum_{1\le a\le q\atop\gcd(a,q)=1}\int_{|\theta|\le\frac{Q}{q g}}f\left(\frac{a}{q}+\theta\right)h\left(-\frac{a}{q}-\theta\right)d\theta
\ll_\ell d Q^3\,.
\label{new6}
\end{align}
Bringing together (\ref{new3}), (\ref{new5}) and (\ref{new6}), we conclude that for $Q<c^{1/3}$
\begin{align}
\int_{\mathfrak M(Q)}f(\alpha)h(-\alpha)=\frac{g}{2}+O_\ell\left(\frac{g}{Q}(\log g)^2+g Q^2 e^{-\kappa_1\sqrt{\log g}}+d Q^3\right).
\label{new7}
\end{align}

\subsection{The asymptotic formula}
It can be concluded from (\ref{new1}), (\ref{new2}) and (\ref{new7}) that for $Q<c^{1/3}$,
\begin{align*}
\psi_{\ell,c,d}=\frac{g}{2}+O_\ell\left(\frac{g}{Q}(\log g)^2+g Q^2 e^{-\kappa_1\sqrt{\log g}}+d Q^3+\frac{g(\log g)^6}{Q^{1/2}}+g^{4/5}(\log g)^6\right)\,.
\end{align*}
We now choose $Q=(\log g)^{14}$. Then we can obtain
\begin{align}
\psi_{\ell,c,d}=\frac{g}{2}+O_\ell\left(\frac{g}{\log g}\right),
\label{new8}
\end{align}
provided that $c\ge(\log g)^{43}$.

For $c\le(\log g)^{43}$, we have
\begin{align*}
\psi_{\ell,c,d}&=\sum_{\substack{\ell c d< n\le g_\ell(c,d)\\ n=\ell c d+c x_0+d y_0\\ 0\le x_0\le d;\,0\le y_0\le c}}\Lambda(n)=\sum_{\substack{0< m\le g\\ m=c x+d y\\x,y\in\mathbb N_0}}\Lambda(\ell c d+m)\\
&=\sum_{1\le y\le c\atop\gcd(y,c)=1}\sum_{m+\ell cd\equiv d y\!\!\!\!\!\pmod c\atop d y\le m+\ell c d\le g}\Lambda(m+\ell c d)+O_\ell (\log g)\\
&=\sum_{1\le y\le c\atop\gcd(y,c)=1}\sum_{n\equiv d y\!\!\!\!\!\pmod c\atop d y+\ell c d\le n\le g+\ell c d}\Lambda(n)+O_\ell (\log g)\\
&=\sum_{1\le y\le c\atop\gcd(y,c)=1}\bigl(\psi(g+\ell c d;c,dy)-\psi(d y+\ell c d;c,dy)\bigr)+O_\ell (\log g).
\end{align*}
Since $c\le (\log g)^{43}\ll (\log d)^{43}$, by the Siegel--Walfisz theorem we have
\begin{align*}
\psi(g+\ell c d;c,dy)-\psi(d y+\ell c d;c,dy)=\frac{g-d y}{\varphi(c)}+O_\ell \left(g e^{-\kappa_2\sqrt{\log g}}\right)\,,
\end{align*}
where $\kappa_2>0$ is an absolute constant.
Thus, for $c\le  (\log g)^{43}$ we have
\begin{align}
\psi_{\ell,c,d}&=\sum_{1\le y\le c\atop\gcd(y,c)=1}\frac{g-d y}{\varphi(c)}+O_\ell\left(g(\log g)^{43}e^{-\kappa_2\sqrt{\log g}}\right)\nonumber\\
&=g-\frac{1}{2}cd+O_\ell\left(g(\log g)^{43}e^{-\kappa_2\sqrt{\log g}}\right)\nonumber
\\
&=g/2+O_\ell\left(g/c+g(\log g)^{43}e^{-\kappa_2\sqrt{\log g}}\right).
\label{new9}
\end{align}
Therefore, from (\ref{new8}) and (\ref{new9}) we know that
$$
\psi_{\ell,c,d}\sim g/2 \quad (\text{as}~ c\rightarrow\infty).
$$

This completes the proof of proposition \ref{prp1}.
\end{proof}

\section{Proof of Theorem \ref{th1}}
\begin{proof}[Proof of Theorem \ref{th1}]
From now on, the symbol $p$ will always denote primes. For $\ell c d<p\le g_\ell(c,d)$, let
$$
\vartheta_{\ell,a,b}(t)=\sum_{\substack{\ell c d<p\le t\\ p=\ell c d+c x_0+d y_0\\ 0\le x_0\le d;\,0\le y_0\le c}}\log p
$$
and
$\vartheta_{\ell,a,b}=\vartheta_{\ell,a,b}\bigl(g_\ell(c,d)\bigr)$.
From Lemma \ref{lem5}, we obtain that
\begin{align}\label{E1}
	\pi_{\ell,a,b}=\sum_{\substack{\ell c d<p\le g_\ell(c,d)\\ p=\ell c d+c x_0+d y_0\\ 0\le x_0\le d;\,0\le y_0\le c}}1=\frac{\vartheta_{\ell,a,b}}{\log g_\ell(c,d)}+\int_{\ell cd}^{g_\ell(c,d)}\frac{\vartheta_{\ell,a,b}(t)}{t\log ^2t}dt
\end{align}
via partial summation.
By the Chebyshev estimate, we have
 $$\vartheta_{\ell,a,b}(t)\leqslant\sum_{p\leqslant t}\log p\ll t,$$
from which it follows that
\begin{equation}\label{E2}
	\int_{\ell cd}^{g_\ell(c,d)}\frac{\vartheta_{\ell,a,b}(t)}{t\log ^2t}dt\ll
	\int_{\ell cd}^{g_\ell(c,d)}\frac{1}{\log^2 t}dt\ll_\ell\frac{g}{(\log g)^2}.
\end{equation}
Again, using the Chebyshev estimate, we have
\begin{equation}\label{E3}
	\vartheta_{\ell,a,b}=\psi_{\ell,a,b}+O_\ell\bigl(\sqrt{g}\bigr).
\end{equation}
Thus, by Proposition \ref{prp1} and Eqs. (\ref{E1}), (\ref{E2}), (\ref{E3}), we conclude that
\begin{equation}\label{E4}
	\pi_{\ell,a,b}=\frac{\psi_{\ell,a,b}}{\log g_\ell(c,d)}+O_\ell\left(\frac{\sqrt{g}}{\log g}+\frac{g}{(\log g)^2}\right)\sim \frac12 \frac{g}{\log g_\ell(c,d)},
\end{equation}
as $c\rightarrow\infty$. Recall that
\begin{equation}\label{E5}
\pi(g_\ell (c,d))\sim \frac{g_\ell(c,d)}{\log g_\ell(c,d)}\sim \frac{(\ell+1)g}{\log g_\ell(c,d)} \quad (\text{as~}c\rightarrow\infty).
\end{equation}

Now, Theorem \ref{th1} follows immediately from (\ref{E4}) and (\ref{E5}).
\end{proof}


\section*{Acknowledgments}

The first named author is supported by National Natural Science Foundation of China  (Grant No. 12201544), Natural Science Foundation of Jiangsu Province, China (Grant No. BK20210784), China Postdoctoral Science Foundation (Grant No. 2022M710121). He is also supported by  (Grant No. JSSCBS20211023 and YZLYJF2020PHD051).

\end{document}